\documentclass[10pt]{amsart}
\usepackage{amsmath, amscd, amsthm} 
\usepackage{amssymb, amsfonts} 
\usepackage{enumerate}
\usepackage{bbm}
\usepackage{color}
 \usepackage{soul}
\usepackage[all]{xy} 
\usepackage[left=1in, right=1.5in, bottom=1.6in]{geometry}
\usepackage{palatino}
\usepackage{mathrsfs}
\usepackage{booktabs} 
\usepackage{aliascnt}
\usepackage{mathrsfs}
\usepackage[svgnames]{xcolor} 
\usepackage{hyperref}
 \definecolor{dark-red}{rgb}{0.4,0.15,0.15}
\hypersetup{
    colorlinks, linkcolor=dark-red,
    citecolor=DarkBlue, urlcolor=MediumBlue
}
\usepackage{mathrsfs}
\usepackage{tikz-cd}
\usepackage[textsize=small]{todonotes}
\usepackage[capitalise]{cleveref}
\usepackage{stackrel}

\renewcommand{\AA}{\mathbb{A}}

\newcommand{\GG}{\mathbb{G}}

\newcommand{\Spec}{\operatorname{Spec}}

\newcommand{\mcal}[1]{\mathcal{#1}}

\renewcommand{\setminus}{\smallsetminus}

\DeclareMathOperator{\Sym}{Sym}

\usepackage[T1]{fontenc}

\numberwithin{equation}{section} 
 \theoremstyle{plain}
 \theoremstyle{definition}
 
\newaliascnt{theorem}{equation}  
\newtheorem{theorem}[theorem]{Theorem}
\aliascntresetthe{theorem}  
  
\newtheorem*{theorem*}{Theorem}
\newaliascnt{prop}{equation}  
\newtheorem{prop}[prop]{Proposition}
\aliascntresetthe{prop}

\newaliascnt{lemma}{equation}  
\newtheorem{lemma}[lemma]{Lemma}
\aliascntresetthe{lemma}

\newaliascnt{corollary}{equation}  
\newtheorem{corollary}[corollary]{Corollary}
\aliascntresetthe{corollary}

\newaliascnt{claim}{equation}  

\aliascntresetthe{claim}

\newaliascnt{conjecture}{equation}  

\aliascntresetthe{conjecture}

\newaliascnt{question}{equation}  
\newtheorem{question}[question]{Question}
\aliascntresetthe{question}

\newaliascnt{defn}{equation}  
\newtheorem{defn}[defn]{Definition}
\aliascntresetthe{defn}

\newaliascnt{example}{equation}  
\newtheorem{example}[example]{Example}
\aliascntresetthe{example}

\theoremstyle{remark}

\newaliascnt{remark}{equation}  
\newtheorem{remark}[remark]{Remark}
\aliascntresetthe{remark}

\newaliascnt{convention}{equation}  

\aliascntresetthe{convention}

\newcommand{\Spc}{\operatorname{Spc}}

\newcommand{\PP}{\mathbb{P}}

\newcommand{\Pic}{\operatorname{Pic}}

\begin{document}
\title{Relative $\AA^{1}$-contractibility of smooth schemes}
\author{Adrien Dubouloz}
\address{CNRS, Université de Poitiers, LMA, Poitiers, France  \& Université Bourgogne Europe, CNRS, IMB UMR 5584, 21000 Dijon, France}
\email{adrien.dubouloz@math.cnrs.fr}

\author{Krishna Kumar Madhavan Vijayalakshmi}
\address{Department of Mathematics ``F. Enriques'', University of Milan, Italy \& 
Université Bourgogne Europe, CNRS, IMB UMR 5584, 21000 Dijon, France }
\email{krishna.madhavan@unimi.it \& krishna.madhavan@ube.fr}

\author{Paul Arne {\O}stv{\ae}r}
\address{Department of Mathematics ``F. Enriques'', University of Milan, Italy}
\email{paul.oestvaer@unimi.it}

\begin{abstract}
We study smooth morphisms $f \colon X \to S$ that are $\mathbb{A}^1$-contractible in the unstable $\AA^1$-homotopy category $\mathcal{H}(S)$. For base schemes $S$ of finite Krull dimension, we show that $\AA^1$-contractibility is a fiberwise property: such a morphism is $\AA^1$-contractible if and only if all its geometric fibers are $\AA^1$-contractible. We apply this criterion to $\AA^n$-fiber spaces, obtaining a geometric description of their $\AA^1$-contractibility in terms of local factorizations as towers of torsors under vector bundles, building on results of Asanuma. In low relative dimensions, we establish rigidity results. In relative dimension $1$, $\AA^1$-contractible morphisms over normal bases are precisely Zariski locally trivial $\AA^1$-bundles. In relative dimension $2$, we show that over bases with characteristic zero residue fields, $\AA^1$-contractible morphisms are $\AA^2$-fiber spaces, and we obtain Zariski local triviality under additional hypotheses on the base. We also exhibit counterexamples in positive and mixed characteristic, and formulate open problems concerning the existence of exotic $\AA^1$-contractible surfaces.
\end{abstract}

\maketitle

\section*{Introduction} \label{sec:intro}

A central theme in $\AA^1$-homotopy theory is to understand how motivic homotopy types reflect the geometry of algebraic varieties. Although $\AA^1$-contractible varieties are indistinguishable from a point in the unstable $\AA^1$-homotopy category, they can exhibit rich and subtle geometry, especially in families. This leads to natural questions: is $\AA^1$-contractibility detected fiberwise for smooth morphisms? What geometric constraints does it impose? And how does it interact with classical structures such as affine fibrations and vector bundle torsors?
\medskip

While these questions have been extensively studied over fields, their relative counterparts remain largely unexplored. In particular, $\AA^1$-contractibility over higher-dimensional bases—such as Dedekind schemes—has not been systematically investigated. One aim of this paper is to develop a structural theory of $\AA^1$-contractible morphisms over a general base.
\medskip

A guiding analogy comes from topology: fiber bundles with contractible fibers are typically homotopy equivalent to the base, and often trivial under mild hypotheses. Thus contractibility tends to eliminate obstructions and impose strong local structure. We show that a similar principle holds in $\AA^1$-homotopy theory: for smooth morphisms, $\AA^1$-contractibility is a fiberwise property.
\medskip

Algebraic varieties that are weakly contractible in the sense of Morel--Voevodsky $\AA^1$-homotopy theory have been intensively studied in recent years. In dimensions $d \geq 3$, this has led to the construction of many smooth $\AA^1$-contractible varieties not isomorphic to affine space, including families with positive-dimensional moduli (see \cite{asok2021A1} for a survey). 
In low dimensions, however, $\AA^1$-contractibility exhibits strong rigidity. The affine line $\AA^1$ is the only $\AA^1$-contractible quasi-projective curve. In dimension $2$, it was only recently shown that, over fields of characteristic zero, the affine plane $\AA^2$ is the unique smooth $\AA^1$-contractible affine surface \cite{choudhury2024}. In positive characteristic, the existence of exotic $\AA^1$-contractible surfaces remains open.
\medskip

In this work, we study the \emph{relative} form of these questions. Given a base scheme $S$, we investigate smooth morphisms $f \colon X \to S$ that are $\AA^1$-contractible over $S$, i.e.\ isomorphisms in the unstable $\AA^1$-homotopy category $\mathcal{H}(S)$. Our starting point is the observation that, under mild hypotheses, $\AA^1$-contractibility is entirely controlled by the fibers.

\medskip

\noindent
\textbf{Main Results.}
\begin{enumerate}
\item \emph{Fiberwise characterization.}  
Let $S$ be a qcqs scheme of finite Krull dimension and let $f \colon X \to S$ be a smooth morphism. Then $f$ is $\AA^1$-contractible if and only if, for every point $s \in S$, the fiber $X_s$ is $\AA^1$-contractible over $\kappa(s)$.

\item \emph{Structure of $\AA^n$-fiber spaces.}  
Let $f \colon X \to S$ be an $\AA^n$-fiber space over a locally noetherian scheme. Under mild hypotheses (e.g.\ $f$ affine or $S$ regular), $X$ admits Zariski locally a factorization as a tower of torsors under vector bundles. 

\item \emph{Low-dimensional rigidity.}
\begin{itemize}
\item In relative dimension $1$, $\AA^1$-contractible morphisms over normal bases are precisely Zariski locally trivial $\AA^1$-bundles.
\item In relative dimension $2$, and over base schemes with characteristic zero residue fields, $\AA^1$-contractibility forces the structure of an $\AA^2$-fiber space. Over Dedekind schemes, such morphisms are Zariski locally trivial provided they are affine. 
\end{itemize}

\item \emph{Pathologies in positive and mixed characteristic.}  
In positive and mixed characteristic, the picture changes markedly. We exhibit $\AA^2$-fiber spaces that are not Zariski locally trivial, showing that natural triviality results for affine fibrations fail in this setting. In addition, we exhibit failures of cancellation: there exist smooth $S$-schemes $X$ with $X\times_S \AA^1_S \cong \AA^3_S$ that are nevertheless not isomorphic to $\AA^2_S$. These pathologies contrast sharply with the rigidity observed in characteristic zero and lead us to formulate open questions concerning the existence of exotic $\AA^1$-contractible surfaces.
\end{enumerate}

\medskip

The first result refines the analogous criterion in stable motivic homotopy theory and relies on a combination of base change techniques and a gluing argument originating in \cite{MV99}. It shows that $\mathcal{H}(S)$ does not detect variation in the isomorphism types of the fibers of a smooth family of $\AA^1$-contractible varieties. Consequently, all known smooth non-isotrivial families of $\AA^1$-contractible varieties of relative dimension $d \geq 3$ over positive-dimensional bases $S$ (see, for example, \cite{DF14,dubouloz2025,Krishna25}) are $\AA^1$-contractible over $S$. 
The second result builds on Asanuma’s work on quasi-polynomial algebras \cite{asanuma82,asanuma1987polynomial} and provides a geometric explanation for the ubiquity of $\AA^1$-contractible morphisms arising from affine fibrations. 
The remaining problems concern, on the one hand, the classification of relatively $\AA^1$-contractible morphisms of relative dimension at most $2$, and, on the other hand, the local structure of such morphisms in the Zariski and Nisnevich topologies.
\medskip

A central theme of the paper is the contrast between low and high dimensions. While dimension $\geq 3$ admits abundant nontrivial families of $\AA^1$-contractible varieties, dimensions $\leq 2$ exhibit strong rigidity. This dichotomy parallels classical phenomena in topology and affine algebraic geometry.
\medskip
\noindent

\textbf{Strategy of proof.}
The fiberwise characterization of $\AA^1$-contractible morphisms is obtained by combining base change properties in the unstable $\AA^1$-homotopy category with a gluing argument along a stratification of the base, ultimately reducing to the case of local schemes. This argument may be viewed as an unstable refinement of analogous results in stable motivic homotopy theory.
\medskip

The structural results for $\AA^n$-fiber spaces rely on Asanuma’s description of affine fibrations in terms of quasi-polynomial algebras, together with geometric input from vector bundle torsors and the Jouanolou--Thomason's trick. This allows us to make the $\AA^1$-contractibility of such morphisms more explicit.
\medskip
S
The low-dimensional classification results combine the fiberwise criterion with classical results in affine algebraic geometry, including the work of Kambayashi, Miyanishi--Sugie, and Russell in dimension $2$, and Sathaye’s theorem on affine fibrations over Dedekind schemes. The analysis of positive and mixed characteristic examples highlights the limitations of these methods and motivates the open problems we formulate.
\medskip

An ingredient in our approach is the use of techniques introduced in \cite{DDO}, which make it possible to study $\AA^1$-homotopy types through the geometry at infinity. These methods relate $\AA^1$-contractibility to the structure of compactifications via the stable $\AA^1$-homotopy type of punctured tubular neighborhoods of the boundary. This viewpoint yields strong geometric constraints, and in particular provides a new mechanism for proving affineness and rigidity results for $\AA^1$-contractible surfaces.

\medskip
\noindent
\textbf{Structure of the paper.}
Section~1 recalls the unstable $\AA^1$-homotopy category and develops the functorial tools needed for our arguments, particularly concerning base change and localization. In Section~2, we study $\AA^n$-fiber spaces and show that their $\AA^1$-contractibility admits a concrete geometric description in terms of iterated vector bundle torsors. Section~3 focuses on low-dimensional phenomena: we obtain a complete picture in relative dimension $1$, and in dimension $2$ we establish rigidity results in characteristic zero while highlighting pathologies in positive and mixed characteristic, leading to several open questions.

\section{Preliminaries}

\subsubsection{Conventions}
Unless otherwise stated, we work over a quasi-compact and quasi-separated base scheme $S$. The term \emph{smooth $S$-scheme} refers to a smooth morphism $f \colon X \to S$ of \emph{finite presentation}. We denote by $Sm_S$ the category of such schemes. Given a morphism of schemes $u \colon T \to S$ and a smooth $S$-scheme $X \to S$, 
we write $X_T$ for the smooth $T$-scheme $X \times_S T \to T$.

\subsection{Elements of \texorpdfstring{$\AA^1$}{A1}-homotopy theory}\label{A1-homotopy theory}
We recall the basic definitions and properties of the $\mathbb{A}^1$-homotopy category $\mathcal{H}(S)$ over a base scheme $S$ that will be used throughout the article. We refer the reader to \cite{MV99, blander2001local, zbMATH02028920, zbMATH05116082, antieau2017primer, wickelgren2020unstable} for further background. 
\medskip

In what follows, $\mathcal{H}(S)$ denotes the homotopy category associated with the model category of motivic spaces over $S$
\[
\mathrm{Spc}_S^{\AA^1} := L_{\AA^1} L_{\mathrm{Nis}} \mathrm{P}(Sm_S),
\]
where $\mathrm{P}(Sm_S)$ denotes the category of presheaves of simplicial sets on $Sm_S$ 
endowed with the projective model structure. 
Here, $L_{\mathrm{Nis}} \mathrm{P}(Sm_S)$ denotes the Nisnevich-local model structure obtained by left 
Bousfield localization with respect to Nisnevich (hyper)covers, 
and $L_{\AA^1}$ denotes the $\AA^1$-localization functor arising from the projection maps
$\AA^1_S \times_S X \to X$, where $X$ ranges over objects of $Sm_S$. 
\medskip

The fibrant objects of $\mathrm{Spc}_S^{\AA^1}$ are called \emph{$\AA^1$-local spaces}, 
and its weak equivalences are called \emph{$\AA^1$-weak equivalences}. 
Throughout, we identify an $S$-scheme with its representing motivic space (the Nisnevich topology is canonical). 
If $X$ and $Y$ are presheaves of simplicial sets on $Sm_S$, we write
\[
[X,Y]_{\AA^1} := \mathrm{Hom}_{\mathcal{H}(S)}(X,Y)
\]
for the set of $\AA^1$-homotopy classes of maps from $X$ to $Y$. 
We now recall additional basic notions.

\begin{defn}\label{A1-rigid}\label{defn:A1-contractibility}
Let $f \colon X \to S$ be a smooth $S$-scheme.
\begin{enumerate}
    \item The \emph{sheaf of $\AA^1$-connected components} of $X$ is the Nisnevich sheaf 
    $\pi_0^{\AA^1}(X)$ on $Sm_S$ 
    associated with the presheaf $U \mapsto [U,X]_{\AA^1}$.  
    We say that $X$ is \emph{$\AA^1$-connected} (over $S$) if the canonical morphism 
    $\pi_0^{\AA^1}(X) \to S$ is an isomorphism of sheaves.
    \item We say that $X$ is \emph{$\AA^1$-rigid} (over $S$) if, for any $U \in Sm_S$, 
    the map $X(U) \to X(\AA^1_U)$ induced by the projection 
    $\mathrm{pr}_1 \colon \AA^1_U = U \times_S \AA^1_S \to U$ is a bijection.  
    \item We say that $X$ is \emph{$\AA^1$-contractible} (over $S$) 
    if $f$ is an $\AA^1$-weak equivalence in $\mathrm{Spc}_S^{\AA^1}$.
\end{enumerate}
\end{defn}

\begin{example}\label{S-point}
Every $\AA^1$-connected smooth scheme $X$ over a locally Henselian scheme $S$ (for instance, the spectrum of a field) admits an $S$-point (see \cite[Remark~2.5]{MV99}).
\end{example}

\begin{example}
A smooth $S$-scheme is fibrant in $\mathrm{Spc}_S^{\AA^1}$ if and only if it is $\AA^1$-rigid. 
Moreover, two $\AA^1$-rigid smooth $S$-schemes are isomorphic as $S$-schemes 
if and only if they are $\AA^1$-weakly equivalent 
(see, e.g., \cite[Lemma~4.1.8 and Corollary~4.1.9]{asok2021A1}).  
For an $\AA^1$-rigid smooth $S$-scheme $X$, the canonical map 
$X \to \pi_0^{\AA^1}(X)$ is an isomorphism of Nisnevich sheaves; 
in particular, such a scheme is not $\AA^1$-connected.  
Examples of $\AA^1$-rigid schemes include the multiplicative group scheme $\mathbb{G}_m$, 
smooth curves of positive genus, and abelian varieties over a field $k$. 
For more examples, see \cite[\S 2]{AM11} or \cite[\S 2.2]{choudhury2024}.
\end{example}

\begin{example}\label{eg:A1-contractibles}
Any Nisnevich locally trivial bundle $f \colon X \to S$
with $\AA^1$-contractible fibers is $\AA^1$-contractible 
(see \cite[Examples~2.3 and~2.4]{MV99}).
\end{example}

\subsection{Functoriality and base change in \texorpdfstring{$\AA^1$}{A1}-homotopy theory}
We recall functorial properties of model categories of motivic spaces that play a central role in understanding how $\AA^1$-contractibility behaves under composition and base change. Standard references for this material include \cite{Ayoub2007six, CD2019} and \cite[\S3]{MV99}; see also \cite[\S4]{hoyois2017six}, \cite[\S2]{zbMATH07566080}, and \cite{bachmann2024stronglya1} for recent perspectives.

\subsubsection{Base change properties}

We follow \cite[Section~3]{oropenbook}.  
Let $u:T\to S$ be a morphism of schemes.  
Pullback along $u$ gives rise to a functor
\[
u_*:\mathrm{P}(Sm_T)\to  \mathrm{P}(Sm_S),\qquad
F\longmapsto (X\longmapsto F(X_T)),
\]
which preserves limits and therefore admits a left adjoint
\[
u^*:\mathrm{P}(Sm_S)\to  \mathrm{P}(Sm_T).
\]
Placing $\mathrm{P}(Sm_S)$ and $\mathrm{P}(Sm_T)$ in their projective model structures, and writing $\mathrm{Spc}_S^{\AA^1}$ and $\mathrm{Spc}_T^{\AA^1}$ for their $\AA^1$-localizations, we obtain Quillen adjunctions
\begin{equation}
\mathrm{P}(Sm_S)
  \stackrel{\overset{u_*}{\longleftarrow}}{\underset{u^*}{\rightarrow}}
\mathrm{P}(Sm_T),
\qquad
\mathrm{Spc}_S^{\AA^1}
  \stackrel{\overset{u_*}{\longleftarrow}}{\underset{u^*}{\rightarrow}}
\mathrm{Spc}_T^{\AA^1}.
\end{equation}

When $u:T\to S$ is smooth, $u^*$ is induced by the forgetful functor $Sm_T\to Sm_S$ obtained by composition with~$u$. Since $u^*$ then preserves limits, it admits a left adjoint
\[
u_\sharp:\mathrm{P}(Sm_T)\to  \mathrm{P}(Sm_S),
\]
and the pair $(u_\sharp,u^*)$ induces additional Quillen adjunctions
\begin{equation}
\mathrm{P}(Sm_T)
  \stackrel{\overset{u^*}{\longleftarrow}}{\underset{u_\sharp}{\rightarrow}}
\mathrm{P}(Sm_S),
\qquad
\mathrm{Spc}_T^{\AA^1}
  \stackrel{\overset{u^*}{\longleftarrow}}{\underset{u_\sharp}{\rightarrow}}
\mathrm{Spc}_S^{\AA^1}.
\end{equation}

These adjunctions control how $\AA^1$-weak equivalences behave under pullback and, when $u$ is smooth, also under pushforward.

\begin{prop}\label{cor:pushpull-A1-weak}
Let $u:T \to S$ be a morphism of schemes. Then:
\begin{enumerate}
  \item  
  If $f:Y\to X$ is a morphism of smooth $S$-schemes that is an $\AA^1$-weak equivalence in $\mathrm{Spc}_S^{\AA^1}$, then the base change
  $f_T:Y_T\to X_T$ is an $\AA^1$-weak equivalence in $\mathrm{Spc}_T^{\AA^1}$.

  \item  
  If $u$ is smooth and $f:Y\to X$ is a morphism of smooth $T$-schemes that is an $\AA^1$-weak equivalence in $\mathrm{Spc}_T^{\AA^1}$, then $f$ is an $\AA^1$-weak equivalence in $\mathrm{Spc}_S^{\AA^1}$.
\end{enumerate}
\end{prop}

\begin{proof}
To prove assertion (1), consider the commutative diagram
\[
\begin{tikzcd}
  \mathrm{P}(Sm_S) \arrow{r}{u^*} \arrow{d}[swap]{\mathrm{id}} &
  \mathrm{P}(Sm_T) \arrow{d}{\mathrm{id}} \\
  \mathrm{Spc}_S^{\AA^1} \arrow{r}{u^*} &
  \mathrm{Spc}_T^{\AA^1}.
\end{tikzcd}
\]
Here $u^*$ sends the presheaves represented by $X$ and $Y$ to those represented by $X_T$ and $Y_T$, and it sends the map induced by $f$ to the map induced by $f_T$.  
Representables are cofibrant in $\mathrm{P}(Sm_S)$ and remain so after $\AA^1$-localization.  
Since $u^*$ is a left Quillen functor, Brown’s lemma \cite[Lemma 1.1.12]{hovey2007model} ensures that $f_T$ is an $\AA^1$-weak equivalence in $\mathrm{Spc}_T^{\AA^1}$.
\medskip

Assertion (2) follows from the fact that $u_\sharp$ preserves representable presheaves and by applying the same argument to the commutative diagram
\[
\begin{tikzcd}
  \mathrm{P}(Sm_T) \arrow{r}{u_\sharp} \arrow{d}[swap]{\mathrm{id}} &
  \mathrm{P}(Sm_S) \arrow{d}{\mathrm{id}} \\
  \mathrm{Spc}_T^{\AA^1} \arrow{r}{u_\sharp} &
  \mathrm{Spc}_S^{\AA^1}.
\end{tikzcd}
\]
\end{proof}

The preceding proposition allows one to transport $\AA^1$-weak equivalences along smooth morphisms.  
In particular, it yields the following compatibility under composition.

\begin{corollary}\label{lem:3-from-2}
Let $f:Y\to X$ and $g:Z\to Y$ be smooth morphisms of smooth $S$-schemes.  
If $g$ is an $\AA^1$-weak equivalence in $\mathrm{Spc}_Y^{\AA^1}$ and 
$f$ is an $\AA^1$-weak equivalence in $\mathrm{Spc}_X^{\AA^1}$, 
then the composite $h=f\circ g:Z\to X$ is an $\AA^1$-weak equivalence in $\mathrm{Spc}_X^{\AA^1}$.
\end{corollary}

\begin{proof}
Since $f$ is smooth, Proposition~\ref{cor:pushpull-A1-weak}(2) implies that $g$ is an $\AA^1$-weak equivalence in $\mathrm{Spc}_X^{\AA^1}$. The assertion then follows from the two-out-of-three property.
\end{proof}

In the setting of \Cref{lem:3-from-2}, the reverse implication 
\[ f,g\ \text{are $\AA^1$-weak equivalences in }\Spc^{\AA^1}_X
\quad \implies \quad
g\ \text{is an $\AA^1$-weak equivalence in }\Spc^{\AA^1}_Y, \]
does not hold in general, as illustrated by the following example (see also \Cref{exa:form-generic-fiber}).

\begin{example}\label{ex:A1-cont-not-factorize}
Consider the smooth morphism of $k$-schemes
\[
g:Z=\AA^2_k\to Y=\AA^1_k,\qquad (x,y)\longmapsto x^2y^2+y.
\]
Both $Z$ and $Y$ are $\AA^1$-contractible over $X=\Spec(k)$, and hence $g$ is an $\AA^1$-weak equivalence in $\mathrm{Spc}_k^{\AA^1}$.  
However, $g$ is not an $\AA^1$-weak equivalence in $\mathrm{Spc}_{\AA^1_k}^{\AA^1}$.  
Indeed, if it were, then \Cref{cor:pushpull-A1-weak}(1), applied to the inclusions $i:s\to \mathbb{A}^1_k$ of points $s$ of $\AA^1_k$, would imply that for every point $s$, the fiber $Z_s$ of $g$ over $s$ is $\AA^1$-contractible over the residue field $\kappa(s)$. However, the fiber over $0$ is the disjoint union of a copy of $\AA^1_k$ and a copy of $\GG_{m,k}$, which is not even $\AA^1$-connected in $\mathrm{Spc}_k^{\AA^1}$.  
Thus $g$ fails to be an $\AA^1$-weak equivalence relative to the base $\AA^1_k$.
\end{example}

A natural question raised by Proposition~\ref{cor:pushpull-A1-weak}(1) is whether the existence of a morphism $u: T \to S$ for which $f_T$ is an $\AA^1$-weak equivalence implies that $f$ itself is an $\AA^1$-weak equivalence. 
The following example illustrates that even faithful flatness of $u$ is insufficient; thus a converse to Proposition~\ref{cor:pushpull-A1-weak}(1) fails in broad generality. 

\begin{example}\label{ex:forms-A1}
Let $k$ be an imperfect field of characteristic $p>0$ with perfect closure $k^{\mathrm{perf}}$, and let 
\[
g:\Spec(k^{\mathrm{perf}})\to \Spec(k)
\]
be the morphism induced by the inclusion $k\subset k^{\mathrm{perf}}$.  
Let $f:X\to\Spec(k)$ be a non-trivial $k$-form of the affine line $\AA^1_k$. Since $\AA^1_k$ has no nontrivial separable forms \cite{russell1970forms}, the base change $X_{k^{\mathrm{perf}}}$ is isomorphic to $\AA^1_{k^{\mathrm{perf}}}$, and hence is $\AA^1$-contractible over $k^{\mathrm{perf}}$. However, $X$ itself is not $\AA^1$-contractible over $k$. Indeed, by \cite{russell1970forms}, either $X(k)=\emptyset$ or $X$ is the complement in $\PP^1_k$ of a closed point $x$ whose residue field $\kappa(x)$ is a purely inseparable extension of $k$. In the first case, non-$\mathbb{A}^1$-contractibility follows from \Cref{S-point}; in the second case, \cite[Theorem~4.1]{achet2017picard} shows that $X$ has a nontrivial Picard group, which again precludes $\AA^1$-contractibility.
\end{example}

Nevertheless, the localization process ensures that $\AA^1$-weak equivalences are always detected Nisnevich locally on the base. This yields the following useful criterion.

\begin{lemma}
\label{lem:A1-cont-local-on-base}
A morphism of smooth $S$-schemes $f:Y\to X$ is an $\AA^1$-weak equivalence in $\mathrm{Spc}_S^{\AA^1}$ if and only if, for every Nisnevich hypercover $U_{\bullet}\to S$, the induced morphism
\[
f_{U_{\bullet}}:Y_{U_{\bullet}}\to X_{U_{\bullet}}
\]
is an $\AA^1$-weak equivalence in $\mathrm{Spc}_{U_{\bullet}}^{\AA^1}$.
\end{lemma}

\begin{proof}
Assume that $f:Y\to X$ is an $\AA^1$-weak equivalence in $\Spc_S^{\AA^1}$, and let $U_{\bullet}\to S$ be any Nisnevich hypercover. By Nisnevich localization, any such hypercover is, by construction, an $\AA^1$-weak equivalence in $\Spc_S^{\AA^1}$. Proposition~\ref{cor:pushpull-A1-weak}(1) then implies that the pullback morphism 
\[
f_{U_{\bullet}}:Y_{U_{\bullet}}\to X_{U_{\bullet}}
\]
is an $\AA^1$-weak equivalence in $\Spc_{U_{\bullet}}^{\AA^1}$.
\medskip

Conversely, assume that for every Nisnevich hypercover $U_{\bullet}\to S$, the pullback $f_{U_{\bullet}}$ is an $\AA^1$-weak equivalence in $\Spc_{U_{\bullet}}^{\AA^1}$. The hypercover $U_{\bullet}\to S$ satisfies Nisnevich hyperdescent for Nisnevich sheaves; that is, for every Nisnevich sheaf $\mcal{F}$ on $S$, the natural map $\mcal{F}(S)\to \mathrm{holim}\,\mcal{F}(U_{\bullet})$ is an isomorphism. Since the $\AA^1$-homotopy sheaves $\pi_n^{\AA^1}$ are (strictly $\AA^1$-invariant) Nisnevich sheaves, it follows by assumption that $f$ induces an isomorphism
\[
\pi_n^{\AA^1}(f_{U_{\bullet}}): \pi_n^{\AA^1}(Y_{U_{\bullet}})\to \pi_n^{\AA^1}(X_{U_{\bullet}}).
\]
Combining this with hyperdescent yields an isomorphism 
\[
\pi_n^{\AA^1}(f): \pi_n^{\AA^1}(Y)\to \pi_n^{\AA^1}(X),
\]
and hence $f$ is an $\AA^1$-weak equivalence in $\Spc_S^{\AA^1}$.
\end{proof}

\subsubsection{Factorization and base change for $\AA^1$-contractible smooth schemes}

We now turn to the behavior of $\AA^1$-contractible smooth schemes under base change and composition.  
We begin with the basic observation that $\AA^1$-contractibility is preserved under arbitrary base change.

\begin{prop}\label{pullbackfunctor}
Let $f:X\rightarrow S$ be a smooth $\AA^1$-contractible $S$-scheme.  
Then for every morphism of schemes $u:T\rightarrow S$, the base change 
$f_T:X_T \rightarrow  T$
is a smooth $\AA^1$-contractible $T$-scheme.
\end{prop}

\begin{proof} 
This is an immediate consequence of \Cref{cor:pushpull-A1-weak}(1). 
\end{proof}

\begin{example}\label{basechange:imperfect}
Let $X$ be a smooth $\AA^1$-contractible scheme over a field $k$.  
For \emph{every} field extension $k\subset k'$, the base change
\[
X_{k'}:=X\times_{\Spec(k)}\Spec(k')
\]
is $\AA^1$-contractible over $k'$.
\end{example}

\smallskip

\begin{example}\label{basechange:closedimm}
Let $i:Z\hookrightarrow S$ be a closed immersion.  
If $f:X\rightarrow S$ is an $\AA^1$-contractible $S$-scheme, then 
\[
\mathrm{pr}_2:X\times_S Z\rightarrow Z
\]
is an $\AA^1$-contractible $Z$-scheme.
\end{example}

The following structural criterion plays a fundamental role in what follows. In the stable $\AA^1$-homotopy category $\mathrm{SH}(S)$ of a noetherian scheme $S$, a smooth morphism $f:X \rightarrow S$ is a stable $\AA^1$-weak equivalence if and only if all its fibers are stably $\AA^1$-contractible over their respective residue fields; see, for example, \cite[Proposition B.3]{bachmann2021norms}.  
A recent observation made in \cite[Proposition~2.1]{ABEH-real-etale} shows that this fiberwise criterion already holds at the unstable level:

\begin{prop}\label{cor:A1-cont-fibers}
Let $S$ be a qcqs scheme of finite Krull dimension. Then, for a smooth morphism $f:X \rightarrow S$, the following are equivalent:
\begin{enumerate}
    \item $f$ is an $\AA^1$-weak equivalence in $\mathrm{Spc}_S^{\AA^1}$;
    \item for every point $s\in S$ with residue field $\kappa(s)$, the fiber $X_s:=X\times_S\Spec(\kappa(s))$ is a smooth $\AA^1$-contractible $\kappa(s)$-scheme.
\end{enumerate}
\end{prop}

\begin{proof}
Since our hypotheses differ slightly from those in \cite[Proposition~2.1]{ABEH-real-etale}, we sketch the argument. The implication $(1)\Rightarrow(2)$ is immediate from \Cref{pullbackfunctor}. For the converse, we argue by induction on $\dim S$. The case $\dim S=0$ reduces to the observation that a smooth morphism 
$X\to S$, where $S$ is the spectrum of a local Artinian ring, is an $\AA^1$-weak equivalence if and only if the pullback $X\times_{S}S_{\mathrm{red}}\to 
S_{\mathrm{red}}$ by the inclusion $i:S_{\mathrm{red}}\to S$ is an $\AA^1$-weak equivalence. This follows from the Morel–Voevodsky localization theorem \cite[Theorem 3.2.21]{MV99}, which asserts that the induced pullback functor $i^\ast\colon \mcal{H}(S)\to \mcal{H}(S_{\mathrm{red}})$ is an equivalence.
\medskip

Now assume that the assertion holds for all schemes of dimension $\le d$, and let $\dim S=d+1$. By \Cref{lem:A1-cont-local-on-base}, it suffices to treat the case where $S=\Spec R$ is local. Let $i:s\hookrightarrow S$ be the closed point of $S$ and let $j:U=S\setminus\{s\}\hookrightarrow S$ be the complementary open immersion. Since $\dim U<\dim S$, the morphism $f_U:X_U\rightarrow U$ is an $\AA^1$-weak equivalence in $\mathrm{Spc}_U^{\AA^1}$ by the induction hypothesis. 
\medskip

Consider the following commutative square in $\mathrm{Spc}_S^{\AA^1}$
\[
\begin{tikzcd}
  j_\sharp X_U \arrow[r] \arrow[d,swap ,"j_\sharp f_U"] 
    & X \arrow[d] \arrow[ddr, bend left=20]\\
  j_\sharp U \arrow[r]  \arrow[drr, bend right=20]
    & j_\sharp U\cup_{j_\sharp X_U} X \arrow[dr] \\  & & i_* X_s.
\end{tikzcd}
\]
in which the upper left square is homotopy co-cartesian.
Since $f_U:X_U\to U$ is an $\mathbb{A}^1$-weak equivalence in $\mathrm{Spc}_U^{\AA^1}$ and $j_\sharp$ is a left Quillen functor, $j_{\sharp}f_U:j_\sharp X_U\to j_\sharp U$ is an $\mathbb{A}^1$-weak equivalence in $\mathrm{Spc}_S^{\AA^1}$. It follows in turn that the map
\[
X\to j_\sharp U\cup_{j_\sharp X_U}X
\]
is an $\AA^1$-weak equivalence. On the other hand, the proof of \cite[Theorem~2.21]{MV99} shows that the natural map 
\[
j_\sharp U\cup_{j_\sharp X_U}X \to  i_* X_s
\]
is also an $\AA^1$-weak equivalence in $\mathrm{Spc}_S^{\AA^1}$. Thus the composite $X\to i_*X_s$ is an $\AA^1$-weak equivalence in $\mathrm{Spc}_S^{\AA^1}$. Since, by assumption, $X_s$ is $\mathbb{A}^1$-contractible in $\mathrm{Spc}_s^{\AA^1}$, the sheaf $i_*X_s$ is $\mathbb{A}^1$-contractible in $\mathrm{Spc}_S^{\AA^1}$, and the result follows.
\end{proof}

\begin{example} \label{thm:Higherdim-Families}
Let $k$ be a perfect field, let $m\geq 0$ and $n\geq 4$ be integers, and let $X$ be the affine variety in $\mathbb{A}^{n-3}_k\times_k \mathbb{A}^{m+4}_k=\mathrm{Spec}(k[a_2,\ldots, a_{n-2}][x_0,\ldots, x_{m},y,z,t])$ defined by the equation 
\[
\underline{x}^ny+z^2+t^3+x_0\bigl(1+\underline{x} +\sum_{i=2}^{n-2} a_i \underline{x}^i\bigr)=0,
\]
where $\underline{x}=\prod_{i=0}^m x_i$. The morphism $f:=\mathrm{pr}_1|_X:X\to S:=\mathbb{A}^{n-3}_k$ is smooth of relative dimension $n-3$, and, by \cite{dubouloz2025,Krishna25}, its fibers $X_s$ over all points $s$ of $S$, closed or not, are $\AA^1$-contractible over the corresponding residue fields $\kappa(s)$ (but not isomorphic to $\AA^{n-3}_{\kappa(s)}$). 

By \Cref{cor:A1-cont-fibers}, $f:X\to S$ is therefore an $\AA^1$-weak equivalence in $\mathrm{Spc}_S^{\AA^1}$. Since $S=\AA^{n-3}_k$ is $\AA^1$-contractible over $k$, it follows from \Cref{lem:3-from-2} that $X$ is an $\AA^1$-contractible smooth $k$-variety of dimension $n+m$. It is not known whether $X$ is isomorphic to $\AA^{n+m}_k$ as a $k$-variety. 
\end{example}

\section{Relative \texorpdfstring{$\AA^1$}{A1}-contractibility of \texorpdfstring{$\AA^n$}{An}-fiber spaces} \label{Asanuma:stabletheorem}

Flat morphisms of finite presentation with fibers isomorphic to affine spaces have long occupied a central place in affine algebraic geometry, notably in connection with the Zariski Cancellation Problem and the Dolgachev--Weisfeiler Problem; see, for instance, \cite{Kr96} for an overview. From the motivic viewpoint, such morphisms furnish a particularly rich source of examples of relatively $\AA^1$-contractible schemes. We begin by recalling the relevant terminology.

\begin{defn}\label{def:A1-fib-space}
An \emph{$\AA^n$-fiber space} over a scheme $S$ is a smooth morphism of finite presentation $f:X\to S$ such that, for every point $s\in S$, the fiber $X_s=X\times_S \Spec(\kappa(s))$ is isomorphic to $\AA^n_{\kappa(s)}$.
\end{defn}

\begin{defn}
An $\AA^n$-fiber space $f:X\to S$ is called a \emph{locally trivial $\AA^n$-bundle} in the Zariski (resp.\ Nisnevich, resp.\ \'etale) topology if every point $s\in S$ admits a Zariski (resp.\ Nisnevich, resp.\ \'etale) neighborhood $U$ such that
\[
U\times_S X \cong U\times_S \AA^n_S
\]
as schemes over $U$.
\end{defn}

\begin{remark}
This notion coincides with what is often called an "$\AA^n$-fibration" or "affine fibration" in the classical literature; see, e.g., \cite{sathaye1983polynomial, asanuma1987polynomial, BhD92, DF10}. Other authors, such as \cite{KM78, KW85}, adopt weaker hypotheses, requiring only that the general closed fibers be isomorphic to $\AA^n$ and that all fibers be geometrically integral. Our terminology is chosen to avoid confusion with topological or model-categorical notions of fibration.
\end{remark}

The following statement is an immediate consequence of \Cref{cor:A1-cont-fibers}.

\begin{theorem} \label{thm:glue-A1-fiber}
An $\AA^n$-fiber space $f:X\to S$ over a qcqs scheme $S$ of finite Krull dimension is an $\AA^1$-weak equivalence in $\mathrm{Spc}_S^{\AA^1}$.
\end{theorem}

The next result, based on Asanuma’s description of $\AA^n$-fiber spaces between noetherian affine schemes \cite[Theorem 4.5]{asanuma82} and \cite[Theorem 3.4]{asanuma1987polynomial}, provides a more precise structural picture, making the relative $\AA^1$-contractibility of such fiber spaces geometrically explicit.

\begin{theorem}\label{thm:An-fiber-space-contractible}
Let $f:X\to S$ be an $\AA^n$-fiber space over a locally noetherian scheme $S$. Assume that either $f$ is affine, or that $f$ is separated and $S$ is regular. Then for every noetherian Zariski affine open subset $U\subset S$ there exist an integer $m$ and a commutative diagram
\[
\xymatrix{ V \ar[d]_{\rho} & \AA^m_U \ar[l]_{\pi} \ar[d]^{\mathrm{pr}_U} \\ X_U \ar[r]^{f_U}  & U}
\]
in which $\rho:V\to X_U$ is a Zariski locally trivial torsor under a vector bundle $F\to X_U$ with affine total space, and $\pi:\AA^m_U\to V$ is a vector bundle.
\end{theorem}

\begin{proof}
First assume that $f$ is affine. For every noetherian affine open subset $U\subset S$, the scheme $X_U=f^{-1}(U)$ is affine and $f_U:X_U\to U$ is an $\AA^n$-fiber space. Since $f_U$ is smooth, the sheaf of relative Kähler differentials $\Omega_{X_U/U}$ is locally free. By \cite[Theorem 4.5]{asanuma82}, there exists a vector bundle $\pi:W\to X_U$ whose total space $W$ is isomorphic over $U$ to $\AA^m_U$ for some $m\geq 0$. The assertion follows by taking $\rho:V\to X_U$ to be the identity.
\medskip

Now assume that $f$ is separated and that $S$ is regular, and let $U=\Spec(A)$ be a Zariski affine open subset of $S$. Since $A$ is regular and $f_U:X_U\to U$ is a smooth separated morphism of finite presentation, the Jouanolou--Thomason homotopy lemma (see, e.g., \cite[Lemma~3.1.4]{asok2021A1}) yields a Zariski locally trivial torsor
\[
\rho:V\to X_U
\]
under a vector bundle $p:E\to X_U$ of rank $r$ with affine total space $V$. The composite $h:=f_U\circ \rho:V\to U$ is a smooth affine morphism of finite presentation.
\medskip

For each point $s\in U$, the induced torsor
\[
\rho_s:V_s\to X_s
\]
is trivial since $X_s$ is affine. Moreover, by the Quillen--Suslin theorem 
\cite{Qui76, zbMATH03822002}, the vector bundle $E_s$ on $X_s\cong \AA^n_{\kappa(s)}$ is trivial. 
It follows that $\rho_s$ is isomorphic to the trivial $\AA^r$-bundle
\[
\mathrm{pr}_2:\AA^r_{\kappa(s)}\times_{\kappa(s)}\AA^n_{\kappa(s)}\to \AA^n_{\kappa(s)},
\]
and hence that $h:V\to U$ is an affine $\AA^{n+r}$-fiber space. The conclusion now follows from the affine case.
\end{proof}

The following example illustrates \Cref{thm:An-fiber-space-contractible} for an $\AA^2$-fiber space $f:X\to S=\AA^2$ with strictly quasi-affine total space.

\begin{example}\label{eg:notZLT}
Let $S=\Spec(\mathbb{Z}[x,y])=\AA^2_{\mathbb{Z}}$, and consider the smooth affine fourfold
\[
Q:=\{xv-yu+z(z-1)=0\}\subset S\times_{\mathbb{Z}}\AA^3_{\mathbb{Z}}
=\Spec(\mathbb{Z}[x,y][z,u,v]).
\]
The projection $\mathrm{pr}_1|_Q:Q\to S$ is smooth and restricts over $S\setminus\{(0,0)\}$ to a Zariski locally trivial $\AA^2$-bundle. The fiber over $(0,0)$ is the disjoint union of two copies of $\AA^2_{\mathbb{Z}}$,
\[
F_0=\{x=y=z=0\},\qquad F_1=\{x=y=0,z=1\}.
\]
Setting $X:=Q\setminus F_1$, the morphism $f=\mathrm{pr}_1|_X:X\to S$ is an $\AA^2$-fiber space whose total space is strictly quasi-affine. In particular, $f$ is not locally trivial in any of the Zariski, Nisnevich, \'etale, fppf, or fpqc topologies. It was observed in \cite{asok2007unipotent} that $X$ is $\AA^1$-contractible over $\Spec(\mathbb{Z})$. A similar argument shows that $f:X\to S$ is an $\AA^1$-weak equivalence in $\mathrm{Spc}_S^{\AA^1}$. Namely, by \cite[\S3]{winkelmann1990free}, there exists a commutative diagram
\[
\xymatrix{ & V=\AA^3_S=\AA^5_{\mathbb{Z}} \ar[dr]^{\mathrm{pr}_S} \ar[dl]_{\rho}\\ X \ar[rr]^{f} \ar[dr] && S=\AA^2_{\mathbb{Z}} \ar[dl] \\ & \Spec(\mathbb{Z}) }
\]
in which $\rho:V=\Spec(k[x,y][t_1,t_2,t_3])\to X$ is the nontrivial $\mathbb{G}_{a,X}$-torsor defined by
\[
(t_1,t_2,t_3)\mapsto (u,v,z)
=\bigl(yt_3-(xt_2-yt_1-1)t_2,\; xt_3-(xt_2-yt_1-1)t_1,\; xt_2-yt_1\bigr).
\]
\end{example}

\begin{remark} \label{rem:Higher-dim}
Higher-dimensional analogues of \Cref{eg:notZLT} were constructed in \cite{asok2007unipotent,ADF2017smooth,ADO23}, providing numerous examples of smooth strictly quasi-affine $\AA^1$-contractible schemes over an affine base $B$ that admit the structure of $\AA^n$-fiber spaces
\[
f_{n,\mathbf{m}}=\mathrm{pr}_{x_1,\ldots,x_n}:X_{n,\mathbf{m}}
=\left\{\sum_{i=1}^n x_i^{m_i}y_i+z(z-1)=0\right\}
\setminus\{x_1=\cdots=x_n=z-1=0\} \to \AA^n_B.
\]
When $B$ is regular, \Cref{thm:An-fiber-space-contractible} implies that $f_{n,\mathbf{m}}$ is an $\AA^1$-weak equivalence in $\mathrm{Spc}_{\AA^n_B}^{\AA^1}$. Since $\AA^n_B$ is $\AA^1$-contractible over $B$, this yields an alternative proof of the $\AA^1$-contractibility of $X_{n,\mathbf{m}}$ over $B$.
\end{remark}

The following two examples illustrate \Cref{thm:An-fiber-space-contractible} for affine $\AA^1$-fiber spaces over singular curves that are not Zariski locally trivial $\AA^1$-bundles.

\begin{example}\label{ex:KambWright}
(See \cite[\S 3.4]{KW85}.) Let $k$ be a field of characteristic zero and let
\[
S=\{x^p-y^q=0\}\subset \AA^2_k=\Spec(k[x,y]),
\]
where $p,q\geq 2$ are relatively prime integers. This defines a rational cuspidal plane curve with normalization $\nu_S:\AA^1_k=\Spec(k[t])\to S$ given by $t\mapsto (t^q,t^p)$. Let $f:X\to S$ be the restriction of the projection $\mathrm{pr}_S:\PP^1_S\to S$ to the complement $X$ of the Zariski closure $Z\cong \AA^1_k$ of the graph $\Gamma\subset \AA^1_S$ of $\nu_S$.
\medskip

By construction, $f:X\to S$ is an affine $\AA^1$-fiber space whose restriction over the complement of the unique singular point $\{(0,0)\}$ of $S$ is a trivial $\AA^1$-bundle. However, $f$ is not a Zariski locally trivial $\AA^1$-bundle: there is no Zariski open neighborhood $U$ of $(0,0)$ over which $f_U:X_U\to U$ is trivial. Indeed, if such a trivialization existed, then the inclusion of $U$-schemes $\AA^1_U\cong X|_U\hookrightarrow \PP^1_U$ would extend to a morphism $\PP^1_U\to \PP^1_U$ sending the singular curve $U=\PP^1_U\setminus \AA^1_U$ birationally onto its normalization $Z|_U=\nu^{-1}(U)$, which is impossible.
\medskip

On the other hand, since $\mathcal{O}_{\PP^1_S}(1)|_X\cong \mathcal{O}_X$ and the restriction to $X$ of the Euler exact sequence
\[
0\to \Omega_{\PP^1_S/S}\otimes \mathcal{O}_{\PP^1_S}(1)\to \mathcal{O}_{\PP^1_S}^{\oplus 2}\to \mathcal{O}_{\PP^1_S}(1)\to 0
\]
splits (as $X$ is affine), it follows that $\Omega_{X/S}\oplus \mathcal{O}_X\cong \mathcal{O}_X^{\oplus 2}$. By \cite[Corollary 4.7]{asanuma82}, this implies that $V=X\times_S\AA^1_S$ is isomorphic to $\AA^2_S$ as a scheme over $S$.
\end{example}

\begin{example}\label{ex:Hamann-example}
(See \cite{Ha75}.) Let $k$ be a field of characteristic $p>0$ and let
\[
S=\{x^{p+1}-y^p=0\}\subset \AA^2_k=\Spec(k[x,y])
\]
be a singular curve with normalization $\nu_S:\AA^1_k=\Spec(k[t])\to S$ given by $t\mapsto (t^p,t^{p+1})$. Let
\[
X=\{u^p+xv^p-v=0\}\subset S\times_k \Spec(k[u,v])
\]
be the non-normal affine surface with singular locus $\{x=y=v-u^p=0\}\cong \AA^1_k$ and normalization
\[
\nu_X:\AA^2_k=\Spec(k[t,w])\to X,\quad (t,w)\mapsto (x,y,u,v)=(t^p,t^{p+1},w-tw^p,w^p).
\]
The morphism $f=\mathrm{pr}_1:X\to S$ is an affine $\AA^1$-fiber space whose restriction over $S^\star:=S\setminus\{(0,0)\}\cong \Spec(k[t^{\pm 1}])$ is the trivial $\AA^1$-bundle $S^\star\times_k\Spec(k[w])$, and whose fiber over $(0,0)$ is $\AA^1_k=\Spec(k[u])$. This $\AA^1$-fiber space is not Zariski locally trivial: otherwise there would exist an affine open neighborhood $U$ of $(0,0)$ such that $f^{-1}(U)\cong \AA^1_U$, in which case the rational function $w=u+x^{-1}yv$ on $X$ would be regular on $f^{-1}(U)$, hence on all of $X$, which is not the case.

On the other hand, the $S$-morphism
\[
\begin{array}{ccc}
\psi\colon V=X\times_{S}\Spec(\mathcal{O}_{S}[z]) & \to & \AA^2_{S}=\Spec(\mathcal{O}_{S}[U,V])\\
(u,v,z) & \mapsto & (U,V)=(u-yz^{p},\,v+z+xz^{p})
\end{array}
\]
is an isomorphism, with inverse
\[
(U,V)\mapsto (u,v,z)=\bigl(U+y(V-U^p-xV^p),\,U^p+xU^{p^2}+yV^{p^2},\,V-U^p-xV^p\bigr).
\]
\end{example}

We conclude with an illustration of \Cref{thm:An-fiber-space-contractible} for an affine $\AA^2$-fiber space between smooth affine schemes that is not a Zariski locally trivial $\AA^2$-bundle.

\begin{example}\label{eg:Asanuma3F}
Let $k$ be a field of characteristic $p>0$ and let $e,s\geq 2$ be integers such that 
$p^e\nmid sp$ and $sp\nmid p^e$. The morphism
\[
\AA^1_k=\Spec(k[u])\to \AA^2_k=\Spec(k[z,t]),\quad u\mapsto (z,t)=(u+u^{sp},u^{p^e})
\]
is a closed immersion with image the curve $C=\{z^{p^e}-t-t^{sp}=0\}$. Let
\[
X=\{x^my-z^{p^e}+t+t^{sp}=0\}\subset \AA^4_{k}=\Spec(k[x,y,z,t]),
\]
where $m\geq 2$. The projection $f=\mathrm{pr}_x:X\to S=\AA^1_k$ is an affine $\AA^2$-fiber space that restricts over $\AA^1_k\setminus\{0\}$ to the trivial $\AA^2$-bundle $(\AA^1_k\setminus\{0\})\times_k\Spec(k[z,t])$, and whose fiber over $0$ is isomorphic to $C\times_k\Spec(k[y])\cong \AA^2_k$.
By \cite[Theorem 5.1]{asanuma1987polynomial}, $f$ is not a Zariski locally trivial $\AA^2$-bundle. On the other hand, by \cite[Theorem 3.4]{asanuma1987polynomial}, the composition $f\circ \mathrm{pr}_1:V=X\times_S \AA^1_S\to S$ is a trivial $\AA^3$-bundle over $S$.
\end{example}

\section{Relative \texorpdfstring{$\AA^1$}{A1}-contractibility of low dimensional smooth schemes }

In this section we study $\AA^1$-contractibility for smooth schemes of relative dimension at most $2$. 
%

\subsection{\texorpdfstring{$\AA^1$}{A1}-contractibility and \texorpdfstring{$\AA^1$}{A1}-rigidity of \'etale schemes}
\label{sect:0-dim schemes}

We begin with the zero-dimensional case.

\begin{prop}\label{0-dim prop}
An \'etale $S$-scheme $f:X\to S$ is $\AA^1$-contractible if and only if $f$ is an isomorphism. 
\end{prop}

\begin{proof}
Every isomorphism is an $\AA^1$-weak equivalence. Conversely, suppose that $f:X\to S$ is an $\AA^1$-weak equivalence in $\mathrm{Spc}_S^{\AA^1}$. By \Cref{cor:pushpull-A1-weak}(1), for every point $s\in S$ with residue field $\kappa(s)$, the base change
\[
f_s:X_s=X\times_S \Spec(\kappa(s))\to \Spec(\kappa(s))
\]
is an $\AA^1$-weak equivalence in $\mathrm{Spc}_{\kappa(s)}^{\AA^1}$. Since $f_s$ is \'etale of finite presentation, $X_s$ is a finite disjoint union of spectra of finite separable extensions of $\kappa(s)$. 
\medskip

Because $X_s$ is $\AA^1$-contractible, it is $\AA^1$-connected, hence connected as a topological space, and by \Cref{S-point} it admits a $\kappa(s)$-rational point. It follows that $X_s\simeq \Spec(\kappa(s))$. Therefore $f$ is bijective with trivial residue field extensions, hence universally injective. Since $f$ is \'etale, it is an open immersion by \cite[\href{https://stacks.math.columbia.edu/tag/02LC}{Tag 02LC}]{stacks-project}; being also surjective, it is an isomorphism.
\end{proof}

The following result provides a relative analogue of \cite[Example 2.1.10]{AM11}.

\begin{prop}\label{etale-A1-rigid}
An \'etale morphism $f:X\to S$ over an integral scheme $S$ is $\AA^1$-rigid.
\end{prop}

\begin{proof}
Let $h:U\to S$ be a smooth $S$-scheme and let $\varphi:\AA^1_U\to X$ be an $S$-morphism. Since $\varphi$ factors through $U$ if and only if its restriction to each connected component $\AA^1_{U'}$ of $\AA^1_U$ factors through the corresponding component $U'$ of $U$, we may assume without loss of generality that $U$ is connected.

Consider the commutative diagram
\[
\xymatrix{
\AA^1_U\times_U \AA^1_U \ar@<-0.5ex>[r]_-{\mathrm{pr}_2} \ar@<0.5ex>[r]^-{\mathrm{pr}_1}
& \AA^1_U \ar[d]_{\mathrm{pr}_U} \ar[r]^{\varphi}
& X \ar[d]^{f} \\
& U \ar[r]^{h}
& S
}
\]
Let $\eta$ be the generic point of $S$ with residue field $K$. Since $f_\eta:X_\eta\to \Spec(K)$ is an \'etale $K$-scheme, it is $\AA^1$-rigid by \cite[Example~2.1.10]{AM11}. It follows that the base change
\[
\varphi_\eta:\AA^1_{U_\eta}\to X_\eta
\]
factors through the projection $\mathrm{pr}_{U_\eta}:\AA^1_{U_\eta}\to U_\eta$. Equivalently, the two morphisms
\[
\varphi\circ \mathrm{pr}_1,\ \varphi\circ \mathrm{pr}_2:\AA^1_U\times_U \AA^1_U\to X
\]
agree on the inverse image of $U_\eta\times_{U_\eta}U_\eta$, hence on a locally closed subset whose image is Zariski dense in $\AA^1_U\times_U \AA^1_U$. Therefore,
\[
\varphi\circ \mathrm{pr}_1=\varphi\circ \mathrm{pr}_2.
\]
By faithfully flat descent for the finitely presented morphism $\mathrm{pr}_U:\AA^1_U\to U$, it follows that $\varphi$ uniquely factors through $\mathrm{pr}_U$, i.e.,
\[
\varphi=\bar{\varphi}\circ \mathrm{pr}_U
\]
for a unique morphism $\bar{\varphi}:U\to X$. This proves that $f:X\to S$ is $\AA^1$-rigid.
\end{proof}

\subsection{\texorpdfstring{$\AA^1$}{A1}-contractibility of smooth schemes of relative dimension 1}
\label{sect:reldim=1}

The classification of smooth proper curves over a field $k$ up to $\AA^1$-weak equivalence is well understood \cite[Proposition 2.1.12]{AM11}: the only $\AA^1$-connected proper curve is $\PP^1_k$, whereas every proper curve of positive genus and every nontrivial $k$-form of $\PP^1_k$ is $\AA^1$-rigid. When $k$ is perfect, a standard completion argument shows that a smooth separated curve over $k$ is $\AA^1$-contractible if and only if it is isomorphic to $\AA^1_k$. In fact, the following stronger statement holds over an arbitrary field.

\begin{prop}\label{1-dim theorem}
An $\AA^1$-contractible smooth curve over a field $k$ is isomorphic to the affine line $\AA^1_k$. 
\end{prop}

\begin{proof}
Let $C$ be a smooth $\AA^1$-contractible curve over $k$. Assume first that $k$ is algebraically closed. Since $C$ is $\AA^1$-connected, it is connected and hence irreducible. Moreover, one has $\mathcal{O}_C(C)^*=k^*$ and $\Pic(C)=0$. We claim that these two properties characterize $\AA^1_k$ among smooth connected curves over $k$.
\medskip

Let $U\subset C$ be a nonempty affine open subset. Its regular completion $\hat{C}$ is a smooth projective curve over $k$, independent of the choice of $U$. The inclusion $U\hookrightarrow \hat{C}$ induces a birational map $\tau:C\dashrightarrow \hat{C}$ which is everywhere defined by the valuative criterion for properness. By Zariski’s Main Theorem \cite[Corollaire 8.12.10]{EGAIV-Grothendieck1966}, the restriction $\tau|_V:V\to \hat{C}$ is an open immersion for every separated open subset $V\subset C$, hence $\tau$ is \'etale.
\medskip

Let $\tilde{C}\subset \hat{C}$ denote the image of $\tau$. The induced morphism $\tilde{\tau}:C\to \tilde{C}$ is surjective and satisfies $\tilde{\tau}_*\mathcal{O}_C=\mathcal{O}_{\tilde{C}}$. Consequently, the pullback maps
\[
\tilde{\tau}^*:\Pic(\tilde{C})\to \Pic(C)
\quad\text{and}\quad
\tilde{\tau}^*:\mathcal{O}_{\tilde{C}}(\tilde{C})^*\to \mathcal{O}_C(C)^*
\]
are injective. Since $\Pic(C)=0$ and $\mathcal{O}_C(C)^*=k^*$, it follows that $\tilde{C}\simeq \AA^1_k$, and that $\tilde{\tau}:C\to \AA^1_k$ is the affinization morphism.
\medskip

If $\tilde{\tau}$ were not an isomorphism, there would exist a point $x\in \AA^1_k$ such that $\tilde{\tau}^{-1}(x)=\{c_1,\dots,c_r\}$ with $r\ge 2$. Then $\mathcal{O}_C(c_1)$ would define a nontrivial invertible sheaf on $C$, contradicting $\Pic(C)=0$. Hence $\tilde{\tau}$ is an isomorphism and $C\simeq \AA^1_k$.

Now let $k$ be arbitrary, and let $\bar{k}$ be an algebraic closure. By \Cref{basechange:imperfect}, the base change $C_{\bar{k}}$ is $\AA^1$-contractible over $\bar{k}$ and hence isomorphic to $\AA^1_{\bar{k}}$ by the previous case. Thus $C$ is a $k$-form of $\AA^1_k$. By \Cref{ex:forms-A1}, any nontrivial form fails to be $\AA^1$-contractible, hence $C\simeq \AA^1_k$.
\end{proof}

Combining \Cref{cor:A1-cont-fibers} with \Cref{1-dim theorem}, we deduce that a smooth morphism $f:X\to S$ of finite presentation and relative dimension $1$ which is an $\AA^1$-weak equivalence is an $\AA^1$-fiber space in the sense of \Cref{def:A1-fib-space}. Conversely, any $\AA^1$-fiber space satisfying the hypotheses of \Cref{thm:glue-A1-fiber} or \Cref{thm:An-fiber-space-contractible} is an $\AA^1$-weak equivalence. The following theorem gives a precise structural characterization.

\begin{theorem}\label{1-dim DD}
Let $f:X\to S$ be a smooth morphism of finite presentation and relative dimension $1$ over a locally noetherian \emph{normal} scheme $S$. The following are equivalent:
\begin{enumerate}
  \item $f$ is an $\AA^1$-weak equivalence in $\mathrm{Spc}_S^{\AA^1}$;
  \item $f$ is an $\AA^1$-fiber space;
  \item $f$ is a Zariski locally trivial $\AA^1$-bundle. 
\end{enumerate}    
\end{theorem}

\begin{proof}
The implication $(3)\Rightarrow (1)$ follows from \Cref{eg:A1-contractibles}, while $(1)\Rightarrow (2)$ is a direct consequence of \Cref{cor:A1-cont-fibers} together with \Cref{1-dim theorem}. It remains to prove $(2)\Rightarrow (3)$. This is essentially due to Kambayashi–Wright \cite{KW85}, except that their argument assumes in addition that $f$ is separated. We briefly recall the proof and indicate how to remove this assumption.
\medskip

We first treat the case $S=\Spec(R)$, where $R$ is a discrete valuation ring. Let $\eta$ and $s$ denote the generic and closed points of $S$, with residue fields $K=\mathrm{Frac}(R)$ and $\kappa$, and let $\pi$ be a uniformizer of $R$. Since $f:X\to S$ is an $\AA^1$-fiber space, the generic fiber $X_\eta$ is isomorphic to $\AA^1_K=\Spec(K[x_0])$. Viewing $x_0$ as a rational function on $X$ and multiplying it by a suitable power of $\pi$, we may assume that $x_0$ extends to a regular function on $X$ whose restriction to the closed fiber $X_s\cong \AA^1_\kappa$ is nonzero. The inclusion $R[x_0]\hookrightarrow \Gamma(X,\mathcal{O}_X)$ then defines an $S$-morphism
\[
h:X\to \AA^1_S=\Spec(R[x_0])
\]
which is an isomorphism over $\eta$.
\medskip

If the restriction $h_s:X_s\to \AA^1_s$ is constant, then by \cite[Proposition~1.4]{KW85} there exist an $S$-endomorphism $\tau:\AA^1_S\to \AA^1_S$ and a quasi-finite $S$-morphism $h':X\to \AA^1_S$, both inducing isomorphisms over $\eta$, such that $h=h'\circ \tau$. For every separated Zariski open subset $U\subset X$, the restriction $h'|_U:U\to \AA^1_S$ is birational and quasi-finite, hence an open immersion by Zariski’s Main Theorem \cite[Corollaire~8.12.10]{EGAIV-Grothendieck1966}. It follows that $h'$ is \'etale and an isomorphism over $\eta$. Since $h'_s:X_s\to \AA^1_\kappa$ is also \'etale, it is an isomorphism. Thus $h'$ is an \'etale bijection with trivial residue field extensions, and therefore an isomorphism. This proves the claim when $\dim S=1$.
\medskip

We now treat the general case by induction on $\dim S$. Since Zariski local triviality is local on the base, we may assume that $S=\Spec(R)$ is a noetherian normal local scheme of dimension $n\geq 2$, with closed point $s$. By the induction hypothesis, the restriction
\[
f_{S^\star}:X_{S^\star}\to S^\star:=S\setminus\{s\}
\]
is a Zariski locally trivial $\AA^1$-bundle. By \cite[\S2.2--2.3]{KW85}, it is in fact trivial, so that $X_{S^\star}\cong \AA^1_{S^\star}$ over $S^\star$. This yields a diagram
\[
\begin{tikzcd}
X \arrow[d,swap,"f"] & X_{S^\star}\cong \AA^1_{S^\star} \arrow[l] \arrow[r] \arrow[d,"\mathrm{pr}_{S^\star}"] & \AA^1_S \arrow[d,"\mathrm{pr}_S"] \\
S & S^\star \arrow[l] \arrow[r] & S
\end{tikzcd}
\]
with horizontal arrows open immersions. It determines a rational $S$-map
\[
h:X\dashrightarrow \AA^1_S
\]
which is an isomorphism over $S^\star$. Since $\AA^1_S$ is affine, $h$ is given by a rational function on $X$ that is regular on $X\setminus X_s$. As $S$ is normal and $f$ is flat, $X$ is normal; moreover, $X_s$ has codimension at least $2$ in $X$. Hence $h$ extends to a morphism $h:X\to \AA^1_S$.
\medskip

By construction, $h$ is an isomorphism over $S^\star$. Thus the canonical morphism
\[
c:h^*\Omega_{\AA^1_S/S}\longrightarrow \Omega_{X/S}
\]
is an isomorphism on $X\setminus X_s$. Both sheaves are invertible, and since $X_s$ has codimension $\geq 2$, it follows from \cite[Théorème~5.10.5]{EGAIV-Grothendieck1966} that $c$ is an isomorphism everywhere. Hence $h$ is \'etale. Since $h$ is an isomorphism over $S^\star$ and $X_s\cong \AA^1_{\kappa(s)}$, the same argument as in the one-dimensional case shows that $h_s$ is an isomorphism, and therefore that $h$ is an isomorphism of $S$-schemes.
\end{proof}

\begin{remark}\label{exa:cusp-base-non-loctrivial}
\Cref{ex:KambWright} and \Cref{ex:Hamann-example} show that the normality assumption in \Cref{1-dim DD} is necessary for the implication $(2)\Rightarrow (3)$.
\end{remark}

We record the following useful criterion for smooth morphisms of relative dimension $1$ over an uncountable algebraically closed field of characteristic zero.

\begin{prop} \label{prop:close-fibers-A1}
Let $k$ be an uncountable algebraically closed field of characteristic zero. A smooth morphism $f:X\to S$ of relative dimension $1$ between algebraic varieties over $k$ is an $\AA^1$-weak equivalence in $\mathrm{Spc}_S^{\AA^1}$ if and only if all its fibers over closed points of $S$ are isomorphic to $\AA^1_k$.
\end{prop}

\begin{proof}
By \Cref{thm:glue-A1-fiber}, it suffices to show that if all closed fibers of $f$ are isomorphic to $\AA^1_k$, then $f$ is an $\AA^1$-fiber space in the sense of \Cref{def:A1-fib-space}. Let $s\in S$ be a (not necessarily closed) point with Zariski closure $Z\subset S$. Replacing $S$ by $Z$, we reduce to the case where $S$ is integral. It is therefore enough to prove that the fiber of $f$ over the generic point of $S$ is isomorphic to $\AA^1_K$, where $K$ is the function field of $S$.
\medskip

Since $X$ and $S$ are of finite type over $k$, there exists a countable subfield $k_0\subset k$, varieties $X_0$ and $S_0$ defined over $k_0$, and a smooth morphism $f_0:X_0\to S_0$ such that $f$ is obtained from $f_0$ by base change along $\Spec(k)\to \Spec(k_0)$. Because $k$ is uncountable, there exists a closed point $s:\Spec(k)\to S$ whose image in $S_0$ is dominant. Equivalently, the composition $\Spec(k)\to S\to S_0$ factors through the generic point $\Spec(K_0)$ of $S_0$, where $K_0$ is the function field of $S_0$. Since the fiber $X_s$ is isomorphic to $\AA^1_k$, it follows that the generic fiber of $f_0$ is a $K_0$-form of $\AA^1_{K_0}$. 
\medskip

As the affine line admits no nontrivial separable forms, this form is trivial, and hence the generic fiber of $f_0$ is isomorphic to $\AA^1_{K_0}$. After base change to $K$, it follows that the generic fiber of $f$ is isomorphic to $\AA^1_K$, as required.
\end{proof}

\medskip

In positive characteristic, the existence of nontrivial purely inseparable forms of the affine line shows that the criterion of \Cref{prop:close-fibers-A1} fails in general, as illustrated by the following example.

\begin{example}\label{exa:form-generic-fiber}
Let $k$ be an algebraically closed field of characteristic $p>0$, and let $X\subset \AA^3_k=\Spec(k[x,y,z])$ be the smooth surface defined by the equation
\[
z^p+xy^p-y =0.
\]
The projection $f:X\to S=\Spec(k[x])$ onto the first coordinate is a smooth morphism whose closed fibers are all isomorphic to $\AA^1_k$. However, the fiber over the generic point of $S$ is a nontrivial $k(x)$-form $Y$ of $\AA^1_{k(x)}$, which becomes trivial after the purely inseparable extension $k(x^{1/p})/k(x)$. In particular, $f$ is not an $\AA^1$-fiber space in the sense of \Cref{def:A1-fib-space}. Since $Y$ is not $\AA^1$-contractible over $k(x)$ (see \Cref{ex:forms-A1}), it follows from \Cref{cor:pushpull-A1-weak}(1) that $f$ is not an $\AA^1$-weak equivalence in $\mathrm{Spc}_{\AA^1_k}^{\AA^1}$.
\end{example}

\smallskip

As a consequence of \Cref{1-dim DD}, we obtain the following characterization.

\begin{corollary}\label{A1unique:DD}
Let $S$ be a noetherian normal scheme. A smooth morphism $f:X\to S$ of finite presentation and relative dimension $1$ is isomorphic to $\AA^1_S$ if and only if $f$ is an $\AA^1$-weak equivalence in $\mathrm{Spc}_S^{\AA^1}$, the relative canonical sheaf $\omega_f=\Omega_{X/S}$ is trivial, and $f$ admits a section.
\end{corollary}

\begin{proof}
The forward implication is immediate: the affine line $\AA^1_S$ is $\AA^1$-contractible over $S$, its sheaf of relative differentials $\Omega_{\AA^1_S/S}$ is trivial, and it admits the zero section $0_S:S\to \AA^1_S$. Conversely, assume that $f:X\to S$ is an $\AA^1$-weak equivalence, that $\omega_f$ is trivial, and that $f$ admits a section. By \Cref{1-dim DD}, $f$ is a Zariski locally trivial $\AA^1$-bundle. Since the automorphism group scheme of $\AA^1_S$ is $\GG_{a,S}\rtimes \GG_{m,S}$, such a bundle is a torsor under a line bundle $L=\mathbb{V}(\mathcal{L})\to S$ for some invertible sheaf $\mathcal{L}$ on $S$. Because $S$ is normal, the pullback homomorphism $f^*:\Pic(S)\to \Pic(X)$ is an isomorphism (see \cite[Theorem~5]{magid1975picard}). The invertible sheaf $\mathcal{L}$ is therefore uniquely determined by the condition $f^*\mathcal{L}\cong \omega_f$. Since $\omega_f\cong \mathcal{O}_X$ by assumption, it follows that $\mathcal{L}$ is trivial, and hence that $f$ is a $\GG_{a,S}$-torsor. The existence of a section then implies that this torsor is trivial, so that $X\cong \AA^1_S$ over $S$.
\end{proof}

\begin{remark}
Since every torsor under a vector bundle over an affine scheme admits a section, \Cref{A1unique:DD} shows that for a noetherian normal \emph{affine} scheme $S$, the properties of being relatively $\AA^1$-contractible and having trivial relative canonical sheaf characterize $\AA^1_S$ among smooth $S$-schemes of finite presentation and relative dimension~$1$.
\end{remark}

As an application, we obtain the following relative version of the Zariski Cancellation Problem for curves (see also \cite{Ha75}).

\begin{corollary}
Let $f:X\to S$ be a smooth morphism of finite presentation and relative dimension $1$ over a normal scheme $S$. If there exists $n\geq 1$ such that
\[
X\times_S \AA^n_S \cong \AA^{n+1}_S,
\]
then $X$ is $S$-isomorphic to $\AA^1_S$.
\end{corollary}

\begin{proof}
From the given isomorphism we obtain
\[
\mathrm{pr}_1^*\omega_f
\;\cong\;
\omega_{X\times_S \AA^n_S/S}
\;\cong\;
\omega_{\AA^{n+1}_S/S}
\;\cong\;
\mathcal{O}_{X\times_S \AA^n_S},
\]
where $\mathrm{pr}_1$ denotes the projection. Since $S$ is normal and $f$ is smooth, $X$ is normal, and the pullback homomorphism
\[
\mathrm{pr}_1^*:\Pic(X)\to \Pic(X\times_S \AA^n_S)
\]
is an isomorphism. It follows that $\omega_f$ is trivial. Moreover, $f$ admits a section obtained by composing the zero section $s_0:S\to \AA^{n+1}_S$ with the projection $\AA^{n+1}_S\cong X\times_S \AA^n_S\to X$. Finally, since both projections
\[
X\times_S \AA^n_S\to X
\quad\text{and}\quad
\AA^{n+1}_S\to S
\]
are $\AA^1$-weak equivalences, it follows that $f:X\to S$ is an $\AA^1$-weak equivalence. The conclusion now follows from \Cref{A1unique:DD}.
\end{proof}

\subsection{\texorpdfstring{$\AA^1$}{A1}-contractibility of smooth schemes of relative dimension 2}
\label{Sect:2-dim schemes}

Relative dimension $2$ is the first setting in which $\AA^1$-contractibility exhibits genuinely new phenomena. Unlike the case of curves, where a complete classification is available, the interaction between $\AA^1$-homotopy theory and the birational geometry of surfaces leads to a much richer and only partially understood picture.
\medskip

Over fields of characteristic zero, recent advances show that $\AA^1$-contractibility is rigid: smooth $\AA^1$-contractible surfaces are isomorphic to $\AA^2$. This result combines techniques from $\AA^1$-homotopy theory with deep structure theorems for affine surfaces.
In positive characteristic, however, several key ingredients break down. In particular, separability issues obstruct the use of $\AA^1$-curves, and existing methods do not rule out the existence of exotic $\AA^1$-contractible surfaces. 
\medskip

In the relative setting, new phenomena already appear over one-dimensional bases. While $\AA^1$-contractibility forces the structure of an $\AA^2$-fiber space under the assumption that the residue fields have characteristic zero, the problem of Zariski local triviality remains widely open beyond special cases (e.g.\ Sathaye’s theorem for Dedekind schemes in characteristic zero). Moreover, counterexamples in mixed characteristic show that naive generalizations fail. These considerations lead to the following guiding questions:
\begin{itemize}
\item Are there smooth affine $\AA^1$-contractible surfaces in positive characteristic that are not isomorphic to $\AA^2$?
\item When is an $\AA^2$-fiber space locally trivial in the Zariski topology?
\item To what extent do cancellation and rigidity phenomena persist in relative dimension $2$?
\end{itemize}

\subsubsection{Smooth $\AA^1$-contractible surfaces over a field}

All known results concerning the characterization of the affine plane $\mathbb{A}^2_k$ among surfaces over a perfect field $k$ (with fixed algebraic closure $\bar{k}$) ultimately rely on the following theorem, which synthesizes classical work of Kambayashi \cite{kambayashi1975absence}, Miyanishi–Sugie \cite{miyanishi1980affine}, and Russell \cite{russell1981affine}.

\begin{theorem} \label{thm:A2-charac}
A smooth separated surface $X$ of finite type over a perfect field $k$ is isomorphic to $\mathbb{A}^2_k$ if and only if it satisfies the following conditions:
\begin{enumerate}
    \item $X_{\bar{k}}$ is connected and connected at infinity, i.e., for every proper completion $X_{\bar{k}} \hookrightarrow Y$, the boundary $Y\setminus X_{\bar{k}}$ is a connected $\bar{k}$-scheme;
    \item $H^0(X_{\bar{k}},\mathcal{O}_{X_{\bar{k}}}^*)=\bar{k}^*$ and $\mathrm{Pic}(X_{\bar{k}})$ is trivial;
    \item there exists a dominant separable morphism $\varphi:\mathbb{A}^1_C\to X_{\bar{k}}$ for some curve $C$.
\end{enumerate}
\end{theorem}

The affine plane $\AA^2_k$ clearly satisfies these conditions. We briefly indicate the idea of the converse. Condition (3) implies that $X_{\bar{k}}$ has negative logarithmic Kodaira dimension $\bar{\kappa}(X_{\bar{k}})$ (see \cite{iitaka1977logarithmic,russell1981affine}). Together with the assumption that $X_{\bar{k}}$ is connected at infinity, results of \cite{miyanishi1980affine,russell1981affine} show that $X_{\bar{k}}$ contains an $\AA^1$-cylinder, i.e., a nonempty Zariski open subset $U$ isomorphic to $\AA^1_Z$ for a smooth affine $\bar{k}$-curve $Z$. 

Assumption (2) then implies that $Z$ is a Zariski open subset of $\AA^1_{\bar{k}}$, and that the projection $\mathrm{pr}_Z:U\to Z$ extends to an $\AA^1$-fiber space $X_{\bar{k}}\to \AA^1_{\bar{k}}$. By \Cref{1-dim DD}, the latter is a trivial $\AA^1$-bundle, and hence $X_{\bar{k}}\cong \AA^2_{\bar{k}}$. The conclusion that $X\cong \AA^2_k$ then follows from \cite{kambayashi1975absence}, which asserts that $\AA^2_k$ admits no nontrivial separable $k$-forms.

\medskip

A series of recent striking results of Choudhury--Roy \cite{choudhury2024, ChRoy26} on the existence of non-constant images of $\AA^1$ in $\AA^1$-connected schemes imply, in particular, that if $X$ is a smooth separated $\AA^1$-connected surface of finite type over a field $k$, then the base change $X_{\bar{k}}$ admits a dominant morphism $\varphi:\AA^1_C\to X_{\bar{k}}$ for some $\bar{k}$-curve $C$. Combining this result with \Cref{thm:A2-charac} yields the following characterization, which slightly strengthens \cite[Theorem~1.1]{choudhury2024}:

\begin{theorem}\label{A2isunique}\label{A2isunique-gen}
A smooth separated surface $X$ of finite type over a field $k$ of characteristic zero is $\AA^1$-contractible if and only if it is isomorphic to $\AA^2_k$.
\end{theorem}

\begin{proof}
The statement is proved in \cite{choudhury2024} under the additional assumption that $X$ is affine; this hypothesis implies, via the Lefschetz hyperplane section theorem, that $X_{\bar{k}}$ is connected at infinity. Our only contribution is the observation that $\AA^1$-contractibility of $X$ already implies that $X$ is affine (see ~\Cref{prop:A1cont-implies-affine} below). 
\medskip

We briefly recall the argument of \cite{choudhury2024} under the affineness assumption. Since $X$ is $\AA^1$-contractible over $k$, its base change $X_{\bar{k}}$ is $\AA^1$-contractible over $\bar{k}$ by \Cref{cor:pushpull-A1-weak}(a). In particular, $X_{\bar{k}}$ is $\AA^1$-connected, hence connected as a scheme, and therefore admits, by the aforementioned result of Choudhury--Roy, a dominant generically finite morphism $\varphi:\AA^1_C\to X_{\bar{k}}$. As $k$ has characteristic zero, this morphism is automatically separable. Moreover, $\AA^1$-contractibility of $X_{\bar{k}}$ implies that $\mathrm{Pic}(X_{\bar{k}})$ is trivial and that $H^0(X_{\bar{k}},\mathcal{O}_{X_{\bar{k}}}^*)=\bar{k}^*$. The conclusion now follows from \Cref{thm:A2-charac}.
\end{proof}

In the proof of \Cref{A2isunique}, we used the following result, which may be viewed as an algebro-geometric counterpart of a classical theorem of Fujita \cite[p.~512, Theorem, item~3]{fujita1982}, asserting that a smooth complex algebraic surface whose analytification is topologically contractible is connected at infinity and, in fact, affine.

\begin{prop} \label{prop:A1cont-implies-affine}
A smooth separated $\AA^1$-contractible surface of finite type over a field $k$ is affine.
\end{prop}

\begin{proof}
It suffices to show that the base change $X_{\bar{k}}$ to an algebraic closure $\bar{k}$ of $k$ is affine. Since $X_{\bar{k}}$ is $\AA^1$-contractible by \Cref{cor:pushpull-A1-weak}(a), we may assume from the outset that $k$ is algebraically closed. We use the notion of the stable $\AA^1$-homotopy type at infinity introduced in \cite{DDO}. Because $X$ is $\AA^1$-contractible in the unstable $\AA^1$-homotopy category $\mathcal{H}(k)$, its $\mathbb{P}^1$-suspension spectrum $\Sigma_{\mathbb{P}^1}^\infty(X,x)$ (for any $k$-rational point $x\in X$) is contractible in the pointed stable $\AA^1$-homotopy category $\mathrm{SH}_\bullet(k)$.
\smallskip

Let $X\hookrightarrow \bar{X}$ be a completion into a proper $k$-scheme. Using resolution of singularities for surfaces, we may assume that $\bar{X}$ is smooth and that the boundary $\partial X=\bar{X}\setminus X$ is a divisor with simple normal crossings. By \cite[Example~4.3.8]{DDO}, the $\AA^1$-contractibility of $\Sigma_{\mathbb{P}^1}^\infty(X,x)$ implies that the stable punctured tubular neighborhood $\mathrm{Tub}^{\times}(\bar{X},\partial X)$ has the same stable $\AA^1$-homotopy type as the $\mathbb{P}^1$-suspension spectrum of $\AA^2_k\setminus\{(0,0)\}$. Using the description of $\mathrm{Tub}^{\times}(\bar{X},\partial X)$ in terms of the irreducible components of $\partial X$ (see \cite[Section~4.2]{DDO}), it follows that $\partial X$ is connected. The same argument as in \cite[p.~512, Theorem, item~3]{fujita1982} then shows that $\partial X$ supports an effective ample divisor on $\bar{X}$, and hence $X$ is affine.
\end{proof}

It follows from the proof of \Cref{A2isunique} that any smooth separated $\AA^1$-contractible surface $X$ of finite type over a perfect field $k$ (with fixed algebraic closure $\bar{k}$) is isomorphic to $\AA^2_k$, provided that $X_{\bar{k}}$ admits a dominant \emph{separable} morphism $\varphi:\AA^1_C\to X_{\bar{k}}$ for some $\bar{k}$-curve $C$. 
\smallskip

As explained above, a principal result of \cite{choudhury2024,ChRoy26} asserts that the $\AA^1$-connectedness of $X_{\bar{k}}$ guarantees the existence of a dominant generically finite morphism $\varphi:\AA^1_C\to X_{\bar{k}}$. However, in positive characteristic, the methods of {\it loc.\ cit.} do not ensure that such a morphism can be chosen to be separable. The following example shows that this issue is genuine: in positive characteristic, there exist smooth $\AA^1$-connected affine surfaces which do not admit any dominant generically finite \emph{separable} morphism $\varphi:\AA^1_C\to X$.

\begin{example} \label{exa:annoying}
Let $k$ be an algebraically closed field of characteristic $p>0$, and consider $\AA^2_k=\Spec(k[x,y])$. Let $\alpha_p=\Spec(k[t]/(t^p))$ act freely on $\AA^2_k$ via
\[
(x,y)\longmapsto \bigl(x+t(p+1)y^p,\; y+t((xy)^p+1)\bigr),
\]
corresponding to the $p$-derivation $\partial=(p+1)y^p\partial_x-((xy)^p+1)\partial_y$ of $k[x,y]$. The kernel $R$ of $\partial$ is generated by
\[
u=x^p,\quad v=y^p,\quad w=(xy+1)^px+y^{p+1},
\]
and the quotient morphism
\[
\pi:\AA^2_k\to X:=\AA^2_k/\alpha_p=\Spec(R)
\]
is an $\alpha_p$-torsor. The surface $X$ is smooth over $k$, and can be identified with the hypersurface in $\AA^3_k=\Spec(k[u,v,w])$ defined by
\[
w^p=(uv+1)^pu+v^{p+1}.
\]
Moreover, the projection $\mathrm{pr}_{u,v}:X\to \AA^2_k$ exhibits $X$ as a purely inseparable finite cover of degree $p$ of $\AA^2_k$, and the composition $\mathrm{pr}_{u,v}\circ \pi:\AA^2_k\to \AA^2_k$ coincides with the Frobenius morphism $(x,y)\mapsto (x^p,y^p)$.

The morphism $\pi:\AA^2_k\to X$ is a purely inseparable universal homeomorphism. In particular, $X$ is $\AA^1$-chain connected, and hence $\AA^1$-connected. Furthermore, $H^0(X,\mathcal{O}_X^*)=k^*$, and since the only $\alpha_p$-linearized invertible sheaf on $\AA^2_k$ is the trivial one with trivial linearization, it follows that $\Pic(X)$ is trivial.

On the other hand, by \cite[Example~1.8.3.2]{ItoMiy21}, a minimal smooth proper completion $\bar{X}$ of $X$ with boundary a simple normal crossing divisor is a projective surface of general type. In particular, the logarithmic Kodaira dimension of $X$ satisfies $\bar{\kappa}(X)=2$. Consequently, $X$ is not isomorphic to $\AA^2_k$, and does not admit any dominant generically finite \emph{separable} morphism $\varphi:\AA^1_C\to X$.
\end{example}

We are currently unable to determine whether the surfaces $X$ constructed in \Cref{exa:annoying} are $\AA^1$-contractible over $k$. In particular, the following question remains open:

\begin{question}
Let $k$ be a perfect field of positive characteristic. Do there exist smooth affine surfaces over $k$ which are $\AA^1$-contractible but not isomorphic to $\AA^2_k$?
\end{question}

\begin{remark}
The affine plane $\AA^2_k$ is known to admit nontrivial forms over any imperfect field $k$. For instance, for every such field $k$, the product $C\times_k \AA^1_k$, where $C$ is a nontrivial purely inseparable form of $\AA^1_k$, defines a purely inseparable $k$-form of $\AA^2_k$. By \Cref{ex:forms-A1}, any such form either has no $k$-rational point or has a nontrivial Picard group, and is therefore not $\AA^1$-contractible over $k$.
\smallskip

In contrast with the case of $\AA^1$, very little is known about the classification of purely inseparable forms of $\AA^2$. For example, it appears to be an open problem whether there exist nontrivial forms of $\AA^2_k$ that admit a $k$-rational point and have trivial Picard group.
\end{remark}

\subsubsection{$\AA^1$-contractibility of smooth schemes of relative dimension $2$ over a base}
\label{sect:reldim=2}

\begin{theorem}\label{thm:Reldim2}
Let $S$ be a scheme whose residue fields have characteristic zero, and let $f:X\to S$ be a smooth separated morphism of finite type and relative dimension $2$. If $f$ is an $\AA^1$-weak equivalence in $\mathrm{Spc}_S^{\AA^1}$, then $f$ is an $\AA^2$-fiber space.
\end{theorem}

\begin{proof}
This follows immediately from \Cref{cor:pushpull-A1-weak}(a) and \Cref{A2isunique}.
\end{proof}

Conversely, \Cref{thm:glue-A1-fiber} and \Cref{thm:An-fiber-space-contractible} provide criteria under which an $\AA^2$-fiber space $f:X\to S$ is an $\AA^1$-weak equivalence in $\mathrm{Spc}_S^{\AA^1}$. Unlike the case of relative dimension $1$, however, the question of Zariski local triviality for $\AA^2$-fiber spaces remains largely open, even when the base scheme $S$ is smooth over a field of characteristic zero.
\smallskip

The only general result in this direction, due to Sathaye \cite{sathaye1983polynomial}, asserts that an affine $\AA^2$-fiber space over a Dedekind scheme with residue fields of characteristic zero is Zariski locally trivial. Combining this with \Cref{thm:Reldim2}, we obtain the following characterization:

\begin{corollary}\label{2-dim DD}
Let $f:X\to S$ be a smooth affine morphism of finite presentation and relative dimension $2$ over a Dedekind scheme $S$ with residue fields of characteristic zero. The following are equivalent:
\begin{enumerate}
    \item $f$ is an $\AA^1$-weak equivalence in $\mathrm{Spc}_S^{\AA^1}$;
    \item $f$ is a Zariski locally trivial $\AA^2$-bundle.
\end{enumerate}
\end{corollary}

\begin{corollary}\label{A2unique:DD}
Let $f:X\to S$ be a smooth affine morphism of finite presentation and relative dimension $2$ over an \emph{affine} Dedekind scheme $S$ with residue fields of characteristic zero. Then $f$ is isomorphic to the trivial $\AA^2$-bundle $\AA^2_S\to S$ if and only if $f$ is an $\AA^1$-weak equivalence in $\mathrm{Spc}_S^{\AA^1}$ and the relative canonical sheaf $\omega_f$ is trivial.
\end{corollary}

\begin{proof}
One implication is immediate. Conversely, by \Cref{2-dim DD}, the morphism $f:X\to S$ is a Zariski locally trivial $\AA^2$-bundle. Since $S$ is affine, it follows from \cite[Theorem~4.9]{BCW76} that $f$ is in fact isomorphic over $S$ to a rank-$2$ vector bundle
\[
p:E=\Spec_S(\Sym^{\bullet}\mathcal{E})\to S
\]
for some locally free sheaf $\mathcal{E}$ of rank $2$ on $S$. As $S$ is one-dimensional, \cite[Theorem~1]{SerreModulesprojectifs} implies that $\mathcal{E}\cong \mathcal{O}_S\oplus \mathcal{L}$ for some invertible sheaf $\mathcal{L}$ on $S$.

Since
\[
\omega_f=\det \Omega_{E/S}\cong p^*\det(\mathcal{E})\cong p^*\mathcal{L},
\]
and the pullback homomorphism $p^*:\Pic(S)\to \Pic(E)$ is an isomorphism, the triviality of $\omega_f$ is equivalent to the triviality of $\mathcal{L}$, and hence to the triviality of $\mathcal{E}$. This shows that $E\cong \AA^2_S$, as desired.
\end{proof}

\begin{corollary}\label{cor:cancel-rel-2}
Let $S$ be an affine Dedekind scheme with residue fields of characteristic zero, and let $f:X\to S$ be a smooth affine $S$-scheme of finite presentation and relative dimension $2$. If there exists $n\geq 0$ such that
\[
X\times_S \AA^n_S \cong \AA^{n+2}_S
\]
as $S$-schemes, then $X\cong \AA^2_S$.
\end{corollary}

\begin{proof}
Since $\omega_{\AA^m_S/S}$ is trivial for every $m\geq 0$, the existence of an isomorphism
\[
X\times_S \AA^n_S \cong \AA^{n+2}_S
\]
implies that
\[
\mathrm{pr}_1^*\omega_f \cong \omega_{X\times_S \AA^n_S/S} \cong \omega_{\AA^{n+2}_S/S}
\]
is trivial. As $\mathrm{pr}_1^*:\Pic(X)\to \Pic(X\times_S \AA^n_S)$ is an isomorphism, it follows that $\omega_f$ itself is trivial.

Moreover, the same isomorphism shows that $f:X\to S$ is an $\AA^1$-weak equivalence in $\mathrm{Spc}_S^{\AA^1}$. The conclusion now follows from \Cref{A2unique:DD}.
\end{proof}

The following variant of \Cref{eg:Asanuma3F} exhibits an affine $\AA^2$-fiber space over the spectrum of a discrete valuation ring whose closed residue field has positive characteristic and which is not a trivial $\AA^2$-bundle. In particular, it illustrates that Sathaye’s result \cite{sathaye1983polynomial} cannot be extended to Dedekind schemes of positive or mixed characteristic, and it also provides a counterexample to \Cref{cor:cancel-rel-2} in that setting.

\begin{example}
Let $S=\Spec(R)$ be the spectrum of a discrete valuation ring with uniformizer $\pi$ and residue field $\kappa=R/\pi R$ of characteristic $p>0$. Let $e,s$ be positive integers such that $sp\nmid p^e$ and $p^e\nmid sp$, and let
\[
X\subset \AA^3_S=\Spec(R[y,z,t])
\]
be the closed subscheme defined by the equation
\[
\pi y - z^{p^e} + t + t^{sp} = 0.
\]
Then, by \cite[Corollary~5.3]{asanuma1987polynomial}, the projection $f=\mathrm{pr}_S:X\to S$ is an $\AA^2$-fiber space which is not isomorphic to the trivial $\AA^2$-bundle $\AA^2_S\to S$, but satisfies
\[
X\times_S \AA^1_S \cong \AA^3_S
\]
as $S$-schemes.
\end{example}


\bibliographystyle{alpha}
\bibliography{DMO}

\end{document}